\documentclass[11pt, reqno]{amsart}
\usepackage{amsmath,amsfonts,amssymb,amsthm}
\usepackage{epsfig}

\usepackage[frenchb,english]{babel}
\usepackage{bbm}
\usepackage{enumerate}
\usepackage{amstext}
\usepackage{xspace}
\usepackage{amsfonts}
\usepackage{amsmath}
\usepackage{amssymb}
\usepackage{amstext}
\usepackage{amsthm}       
\usepackage{xspace}
\usepackage[active]{srcltx}
\usepackage{url}
\usepackage{hyperref}

\newcommand{\R}{\mathbb R}

\newcommand{\N}{\mathbb N}

\newcommand{\PP}{\mathbb P}

\newcommand{\A}{\mathcal{A}}

\renewcommand{\S}{\mathcal{S}}

\newcommand{\tagliato}{$\kern-5 mm -$}
\newcommand{\tagliat}{$\kern-4 mm -$}

\newcommand{\MM}{\mathcal M}

\newtheorem{Th}{Theorem}[section]
\newtheorem*{Th*}{Theorem}
\newtheorem{prop}[Th]{Proposition}
\newtheorem{lemma}[Th]{Lemma}
\newtheorem{df}[Th]{Definition}

\newtheorem{nota}{Notation}[section]
\newtheorem{guess}[Th]{Remark}
\newtheorem{example}[Th]{Example}

\newenvironment{oss}{\begin{guess} \begin{rm}}{\end{rm} \end{guess}}

\begin{document}

\title{Convergence of solutions for some degenerate discounted Hamilton--Jacobi equations}
\author{Maxime Zavidovique}
\address{
IMJ (projet Analyse Alg\' ebrique), UPMC,  
4, place Jussieu, Case 247, 75252 Paris C\' edex 5, France}
\email{zavidovique@math.jussieu.fr} \keywords{Discounted  Hamilton--Jacobi equations, viscosity solutions, weak KAM Theory, Mather measures}
\subjclass[2010]{}

\maketitle

\begin{abstract}
We study solutions of Hamilton--Jacobi equations of the form 
$$\lambda \alpha(x) u_\lambda(x) + H(x, D_x u_\lambda) = c,$$
 where $\alpha$ is a nonnegative function, $\lambda$ a positive constant, $c$ a constant and $H $ a convex coercive Hamiltonian. Under suitable conditions on $\alpha$ we prove that the functions $u_\lambda$ converge as $\lambda\to 0$ to a function $u_0$ that is a solution of the critical equation $H(x, D_x u_0) = c$.
\end{abstract}

%
%
\section*{Introduction}
The use of discounting in optimal control and Hamilton--Jacobi equations theory is both motivated by the models used and very useful from a theoretical point of view. In the models, the discount accounts for the lesser influence of events far away in time. From the theoretical point of view, the discount brings exponential terms $e^{-\lambda t}$ allowing to consider infinite horizon problems as  improper integrals converge, and in Hamilton--Jacobi equations, it allows to prove strong comparison principle (or to make some solution operators strongly contracting) and are followed by powerful existence and uniqueness results.     

To our knowledge, one of the first striking applications of this method was for homogenization purposes in the famous \cite{LPV}. The authors consider a periodic (in the first variable) coercive (in the second variable) Hamiltonian $H : \R^N \times \R^N \to \R$ and for $\lambda>0$ prove that there exists a unique periodic viscosity solution $u_\lambda : \R^N \to \R$ to
$$\lambda u(x) +H(x,D_x u) = 0.$$
They then prove that the $(u_\lambda)_\lambda$ are equiLipschitz and that $\lambda u_\lambda$ converges uniformly to a constant, $-c_H$, as $\lambda \to 0$. Using the Ascoli Theorem, the authors conclude to the existence of a periodic viscosity solution to the cell problem
$$H(x,D_xu) = c_H.$$
 
 This last point uses an extraction and the problem of actual convergence of the family $u_\lambda$ was not tackled for quite some time. Following a first breakthrough in this direction \cite{IS}, a general convergence Theorem was established, under an additional convexity hypothesis in \cite{DFIZ}. Since then there has been a huge activity on this problem and therefore an important subsequent literature. Let us state amongst others \cite{DFIZ2} for a discrete time version, \cite{ISico} for a non compact version, \cite{PSnet} on networks, \cite{CCIZ,WYZ,Chen} for nonlinear discounting versions, \cite{I21,I22} for second order PDE versions, \cite{DZsys} for weakly coupled systems, following ideas from \cite{DSZ} and then widely generalized in \cite{Isys1,Isys2}, finally let us mention \cite{A12} for a more geometric interpretation of the convergence result and references therein.

Such convergence results also have limitations. All the previous require some convexity assumptions and a striking counterexample has been constructed in \cite{Zi} when otherwise. Moreover, all the previous works require that the discounting be increasing in some sort and counterexamples to the convergence result appear in \cite{CCIZ} when this monotonicity is not strong enough. 

The goal of this paper is to provide a setting with non--decreasing discounting where convergence still holds. More precisely we study equations of the form 
$$\lambda \alpha(x)u(x) +H(x,D_x u) = c_H,$$
where $\alpha$ is a nonnegative function that may vanish at some places. From  an economical point of view, this may provide a model for settings in which interest rates depend on the space variable and are allowed to vanish at some places. 
From a theoretical point of view, if $\alpha$ is positive, we will see in Proposition \ref{bb2} that the study reduces to the previous case. If $\alpha$ is identically $0$ on the contrary, not much can be said. We require in this paper that $\alpha$ is positive on the Aubry set of $H$ that plays a major role in the study of the limiting equation $H(x,D_xu) = c_H$. Our result is then that the solutions $u_\lambda$ to the discounted equations converge when $\lambda \to 0$. To obtain such a result, the main difficulty is to obtain quantitative properties on the behavior of characteristic trajectories associated to $u_\lambda$. This is done in Proposition \ref{excursion}. 

\subsection*{Organization of the paper}
In section \ref{setting},  we introduce the setting, the main hypotheses and objects of study and state precisely our main result.

In section \ref{generality}, we state well known, or folklore results for general Hamilton--Jacobi equations that we later apply to our particular setting.

In section \ref{sec:weakKAM} we recall main features and tools of weak KAM theory and Aubry--Mather theory. Those are at the core of the proof of the main result.

In section \ref{sec:deg} we study our degenerate discounted equations. In particular, we prove they satisfy a strong comparison principle (Proposition \ref{comp-l}) and in particular that there exists indeed a unique solution $u_\lambda$ for each fixed
 $\lambda>0$. We then derive from classical weak KAM theory representation formulas of Lax--Oleinik type for $u_\lambda$ (in Theorem \ref{representation} and \ref{representation-l}) and we establish the main technical lemma about minimizing trajectories for $u_\lambda$ (Proposition \ref{excursion}).

Finally, in section \ref{sec:proof} we prove the convergence result.

To conclude, in section \ref{formula} we establish an alternate expression for the limit $u_0$. For readers familiar with \cite{DFIZ}, this formula will not come as a surprise.

\subsection*{Acknowledgement} The authors wishes to thank the anonymous referees for their help in improving the motivations and presentation of the present work.

\section{Setting and main result}\label{setting}
Let $M$ be a closed connected smooth compact manifold. We denote by $TM$ and $T^*M $ respectively the tangent and cotangent bundles of $M$.

We endow $M$ with a Riemannian metric $g$ and denote by $d : M\times M \to \R$ the associated distance. As $M$ is compact, all such Riemannian metrics are equivalent and all our results are independent of this choice. If $x\in M$ we will denote by $\|\cdot \|_x$ the norm associated to $g$, either on $T_xM$ or on $T^*_x M$. We will denote by $\pi$ the canonical projections from $TM\to M$ and from $T^*M\to M$, $(x,v),(x,p)\to x$ according to the context. If $A$ is any set and $f : A\to \R$ is a bounded function we denote by $\|f\|_\infty = \sup\limits_{x\in A} |f(x)|$ its sup norm.

Unless explicitly specified, when a sequence of functions is convergent, it will always be for the sup norm.

We will consider a Hamiltonian function $H : T^*M \to \R$ that we will always assume to be continuous. Moreover we will assume that 
\begin{itemize}
\item[(H1)]\label{H1}(Convexity) For every $x\in M$ the function $H(x,\cdot) : T_x^*M \to \R$ is convex.
\item[(H2)]\label{H2}(Coercivity) $H(x,p) \to +\infty$ as $\|p\|_x \to +\infty$.
\end{itemize}
Note that thanks to the compactness of $M$ and the convexity of $H$, it is equivalent to require that coercivity holds pointwise for all $x\in M$ or uniformly on $M$. 

We will at some point enforce (H2) by the stronger condition
\begin{itemize}
\item[(H2')]\label{H2'}(Superlinearity) $H(x,p)/\|p\|_x \to +\infty$ as $\|p\|_x \to +\infty$.
\end{itemize}
Again, this limit is equivalently verified pointwise or uniformly in $x$.

Associated to $H$ is the critical value $c_H\in \R$, it is the only real number such that the critical equation 
\begin{equation}\label{HJ-crit}
H(x,D_x u) = c_H, \quad \quad x \in M \tag{HJ-crit},
\end{equation}
admits viscosity solutions\footnote{All notions and definitions related to viscosity will be given in the next section.}. An important object to study such solutions is the projected Aubry set $\A \subset M$ (that will be precisely defined later). It is closed, compact and non--empty. Moreover one of its fundamental properties  is provided by the following theorem (\cite[Theorem 6.2]{FS05}):
\begin{Th}\label{strict}
There exists a continuous function $v : M\to \R$ that is a viscosity subsolution of \eqref{HJ-crit}, that is $C^{\infty}$ and strict on $M\setminus \A$ meaning that
$$\forall x\in M\setminus \A , \quad H(x,D_x v) < c_H.$$
\end{Th}

Finally, let us introduce the function we will use to discount the Hamilton--Jacobi equation: $\alpha : M\to \R$ is a continuous function that verifies
\begin{itemize}
\item[($\alpha$1)]\label{a1}(Non--negativity) The function $\alpha$ is nonnegative on $M$.
\item[($\alpha$2)]\label{a2}(Positivity) The function $\alpha$ is positive on $\A$.
\end{itemize}

The discounted equations we will be studying are the following: given a positive constant $\lambda >0$
\begin{equation}\label{HJ-l}
\lambda \alpha(x) u(x) + H(x,D_x u) = c_H, \quad \quad x \in M \tag{HJ-$\lambda$}.
\end{equation}

As we will see in Proposition \ref{comp-l}, given $\lambda >0$, the previous equation admits a unique viscosity solution $u_\lambda$. Our main theorem is the following:

\begin{Th}\label{main}
The family $(u_\lambda)_{\lambda >0}$ uniformly converges, as $\lambda \to 0$ to a function $u_0 : M\to \R$ that is a viscosity solution to 
$$ H(x,D_x u_0) = c_H, \quad \quad x \in M.$$
\end{Th}

Following from the proof we actually will establish two characterizations of the limit function $u_0$, the first one in terms of critical subsolutions and Mather measures (Definition \ref{limit}). The second one involves Mather measures and Peierl's barrier (Definition \ref{limit2}).

\section{Generalities on Hamilton--Jacobi equations} \label{generality}

Equations \eqref{HJ-l} and \eqref{HJ-crit} fall into the scope of general Hamilton--Jacobi equations of the type 
\begin{equation}\label{HJgen}
G\big(x,D_xu,u(x)\big) = 0, \quad \quad x \in M , 
\end{equation}
where $G : T^*M \times \R \to \R$ is continuous, convex in the last two variables,  coercive in the second variable and non--decreasing in the last variable. Unless explicitly stated otherwise, we assume in this section that $G$ verifies the previous hypotheses. Let us state some well known facts about those. The general references for this section are \cite{bardi,barles,Fa,Sic}. Note that many results below are not stated with optimal hypotheses but they will suffice our needs.

\begin{df}\rm
A continuous function $u : M\to \R$ is a 
\begin{itemize}
\item viscosity subsolution to \eqref{HJgen} if for all  $C^1$ function $\varphi : M\to \R$ and $x\in \R$ such that $u-\varphi$ has a local maximum at $x$, $G\big(x,D_x\varphi,u(x)\big) \leqslant  0$;
\item  viscosity supersolution to \eqref{HJgen} if for all  $C^1$ function $\varphi : M\to \R$ and $x\in \R$ such that $u-\varphi$ has a local minimum at $x$, $G\big(x,D_x\varphi,u(x)\big) \geqslant  0$;
\item viscosity solution if it is both a viscosity subsolution and supersolution.
\end{itemize}
\end{df}

In this article, even if omitted, all functions considered will be continuous and all subsolutions, supersolutions or solutions will always be in the viscosity sense as above, hence the adjective will sometimes be omitted.

The next Proposition is referred to as stability in viscosity theory:
\begin{prop}\label{stab}
A uniform limit of subsolutions (resp. supersolutions, solutions) of \eqref{HJgen} is a subsolution  (resp. supersolutions, solutions) of \eqref{HJgen}.
\end{prop}

It is also worth noticing that for $C^1$ functions, it is equivalent to be subsolution, supersolution, solution in the viscosity sense or in the classical sense. 

We start by an important consequence of the coercivity assumption (\cite[Lemme 2.5]{barles}). Note that the proof being local, it adapts to our manifold setting in a straightforward manner (see also \cite[Theorem 7.5.2]{Fa}).

\begin{prop}\label{lip}
Let $u : M\to \R$ be a subsolution of \eqref{HJgen}. Then $u$ is automatically Lipschitz. A Lipschitz constant of $u$ is given by
$$K=\max\big \{ \|p_x\| , \ \ (x,p,r)\in T^*M\times[-\|u\|_\infty, \|u\|_\infty] \cap G^{-1}\big( (-\infty,0)\big)\big\}.$$
\end{prop}
As $G$ is also convex, we derive the following classical consequence (remembering that by Rademacher's theorem, Lipschitz functions are automatically differentiable almost everywhere):
\begin{prop}\label{aesub}
Let $u:M\to \R $ be a continuous function, then the following assertions are equivalent:
\begin{itemize}
\item $u$ is a subsolution of \eqref{HJgen};
\item $ u$ is Lipschitz and verifies the inequality $G\big(x,D_xu,u(x)\big) \leqslant 0$ for almost every $x\in M$.
\end{itemize}

\end{prop}

Let us now state some general properties on sub and supersolutions of \eqref{HJgen}. The first item is general, the last two rely on our convexity hypothesis.

\begin{prop}\label{viscositygen}
The following properties hold:
\begin{enumerate}
\item If $u\in C^0(M,\R)$ is a pointwise supremum (resp. infimum) of a family of subsolutions (resp. supersolutions) of \eqref{HJgen}, then $u$ is a subsolution (resp. supersolution) to \eqref{HJgen}.
\item If $u$ is a pointwise infimum of a family of equi--Lipschitz subsolutions of \eqref{HJgen}, then $u$ is itself a subsolution of \eqref{HJgen}.
\item If $u$ is a convex combination of a family of equi--Lipschitz subsolutions of \eqref{HJgen}, then $u$ is itself a subsolution of \eqref{HJgen}.

\end{enumerate}

\end{prop}

Here is an approximation theorem that we will use repeatedly. Its proof follows from standard mollification technics. We will see how to obtain it from similar known results.

\begin{prop}\label{approx}
Let $G : T^*M\times \R \to \R$ be continuous, convex in $p$ and let $u : M\to \R$ be a Lipschitz subsolution of \eqref{HJgen}. Then for all $\varepsilon>0$ there exists a smooth $u_\varepsilon : M\to \R$ such that  $G\big(x,D_xu_\varepsilon,u_\varepsilon(x)\big) \leqslant \varepsilon$ and $\|u-u_\varepsilon\|_\infty < \varepsilon$.
\end{prop}

\begin{proof}
Let $\varepsilon>0$. Let $K>0$ be a Lipschitz constant for $u$. If $\eta>0$ we may apply \cite[Theorem 8.1]{fatmad}, see also \cite[Lemma 2.2]{DFIZ} and \cite{FS05}, to the Hamiltonian $F(x,p)=G\big(x,p,u(x)\big)$ to obtain a function $v_\eta : M\to \R$ that is smooth, $2K$--Lipschitz, such that $\|u-v_\eta\|_\infty <\eta$ and $F(x,D_xv_\eta)<\eta$ for all $x\in M$.  

It now follows from the uniform continuity of $G$ on  the compact set $\{(x,p,r)\in T^*M\times \R , \ \ \|p\|_x\leqslant 2K ,\ \ |r|\leqslant \|u\|_\infty +1\}$ that for $\eta$ small enough, the function $v_\eta$ verifies the requirements of the proposition. 
\end{proof}

We conclude by the most important result we will need which is a strong comparison principle. It is rather folklore (see \cite[Theorem 2.7]{barles} and the following discussion). Let us however provide elements of proof for the reader's convenience:
\begin{Th}\label{compG}
Let $G : T^*M \times \R \to \R$ be continuous, convex in the last two variables,  coercive in the second variable and non--decreasing in the last variable. Assume that there exists a $C^1$ function $\varphi : M\to \R$ such that 
$G\big(x,D_x\varphi ,\varphi(x)\big) < 0$, for all $x\in M$. Let $u : M\to \R$ and $v: M\to \R$ be respectively a subsolution and a supersolution of \eqref{HJgen}. Then $u\leqslant v$. 
\end{Th}

\begin{proof}[sketch of proof]

If $\eta\in (0,1)$ we define $u_\eta = (1-\eta)u+\eta \varphi$. Note that thanks to Proposition \ref{lip}, the $u_\eta$ are equi--Lipschitz. We denote by $K>0$ a common Lipschitz constant. We will prove that $u_\eta\leqslant v$ for all $\eta\in (0,1)$ which proves the Theorem letting $\eta \to 0$.

We now fix $\eta\in (0,1)$ until the end of the proof. By compactness, there exists a positive constant $r>0$ such that $G\big(x,D_x\varphi ,\varphi(x)\big) < -r$, for all $x\in M$. Hence, thanks to the convexity hypothesis of $G$, $u_\eta$ is a subsolution of 
$$G\big(x,D_xu_\eta,u_\eta(x)\big) \leqslant -\eta r, \qquad \qquad \qquad x\in M.$$

If $0<\varepsilon <\eta r$ we may apply Proposition \ref{approx} to find a smooth function $\tilde u_\varepsilon$ such that $\|u_\eta - \tilde u_\varepsilon \|_\infty \leqslant \varepsilon$  and 

$$G\big(x,D_x\tilde u_\varepsilon,\tilde u_\varepsilon(x)\big) \leqslant -\eta r+\varepsilon <0, \qquad \qquad \qquad x\in M.$$

Let now $x_\varepsilon \in M$ be a maximum point of $\tilde u_\varepsilon - v$. We may see $\tilde u_\varepsilon$ as a test function for $v$ and the supersolution criterion yields that 
$$G\big(x_\varepsilon, D_{x_\varepsilon} \tilde u_\varepsilon, v(x_\varepsilon)\big) \geqslant 0.$$
Moreover, as $G\big(x_\varepsilon,D_{x_\varepsilon}\tilde u_\varepsilon,\tilde u_\varepsilon(x_\varepsilon)\big) \leqslant -\eta r+\varepsilon <0$, we deduce that $\tilde u_\varepsilon (x_\varepsilon) -v(x_\varepsilon) \leqslant 0$ as $G$ is non--decreasing in the last variable. Finally, by the choice of $x_\varepsilon$ it follows that $\tilde u_\varepsilon\leqslant v$ on $M$ and letting $\varepsilon \to 0$ we conclude that $u_\eta\leqslant v$ hence the Theorem. 

\end{proof}

\section{Classical weak KAM theory}\label{sec:weakKAM}

Good references for this section are \cite{FS05,Fa12}. Much of the present content is also recalled in \cite{DFIZ}.
\subsection{For coercive Hamiltonians}

We here study stationary equations associated to a continuous Hamiltonian function $H : T^*M\to \R$  that verifies hypotheses (H1) and (H2). If $c\in \R$ we consider the Hamilton--Jacobi equation 
\begin{equation}\label{HJs}
H(x,D_x u) = c , \qquad \qquad x\in M.
\end{equation} 

If $u$ is a $C^1$ function then it is a subsolution to \eqref{HJs} as soon as 
$$c\geqslant \max \{H(x,D_x u),\ \ x\in M\}.$$

The equation \eqref{HJs} falls into the scope of the previous section hence all its results apply.

In particular, all subsolutions to \eqref{HJs} are automatically Lipschitz with a common Lipschitz being 
\begin{equation}\label{kappa}
\kappa_c =  \max\big\{ \|p\|_x, \ \ (x,p)\in H^{-1}\big((-\infty, c ) \big)\big\}.
\end{equation}

This leads us to the important notion of critical value:

\begin{df}\rm
The critical value $c_H$ is the smallest constant $c\in \R$ for which \eqref{HJs} admits subsolutions.
\end{df}

It is not hard, using the Ascoli theorem and stability of viscosity subsolutions that such a minimum is actually attained. As already introduced, we will call critical equation the stationary Hamilton--Jacobi equation 
\begin{equation}\label{HJ-crit}
H(x,D_x u) = c_H, \quad \quad x \in M \tag{HJ-crit},
\end{equation}
We will also refer to as critical the subsolutions, supersolutions and solutions to \eqref{HJ-crit}.
Hence in the rest of the paper, a critical solution (resp. critical subsolution, critical supersolution) is a solution (resp. subsolution, supersolution) to \eqref{HJ-crit}.
 Note finally that $c_H$ is the only constant for which \eqref{HJs} admits solutions, called critical solutions or weak KAM solutions.

As the set of subsolutions to \eqref{HJ-crit} is equi--Lipschitz, only the restriction of $H$ to a compact subset of $T^*M$ \big(namely $H^{-1}\big((-\infty,c_H)\big)$\big) is relevant. Hence it is not too restrictive to consider superlinear Hamiltonians, up to modifying the initial coercive Hamiltonian outside of this compact set.

\subsection{For superlinear Hamiltonians}\label{superlinear}

In the rest of this section, we will enforce condition (H2) in condition (H2') assuming that $H$ is moreover superlinear and still verifies (H1). This section is more dynamical in spirit. However as the Hamiltonian is not smooth, there is no Hamiltonian flow at hand and other arguments are required (see also \cite{DS, DavZav} and the appendix in \cite{DZ15}).

We define the Lagrangian $L : TM\to \R$ through the Legendre transform:

$$\forall (x,v)\in TM, \quad L(x,v) = \max_{p\in T^*_xM} p(v)-H(x,p).$$ 
The Lagrangian $L$ enjoys similar properties as $H$. It is continuous and verifies 
\begin{itemize}
\item[(L1)]\label{L1}(Convexity) For every $x\in M$ the function $L(x,\cdot) : T_xM \to \R$ is convex.
\item[(L2')]\label{L2'}(Superlinearity) $L(x,p)/\|v\|_x \to +\infty$ as $\|v\|_x \to +\infty$.
\end{itemize}
Some further properties are given by the proposition below:
\begin{prop}\label{fenchel}
The Hamiltonian and Lagrangian are linked by the following:
\begin{itemize}
\item For all $(x,v)\in TM$ and $p\in T^*_x M$, the inequality $L(x,v)+H(x,p) \geqslant p(v)$ holds and is called the Fenchel inequality.
\item Equality, $L(x,v_0)+H(x,p_0) = p_0(v_0)$, holds  if and only if one of two equivalent properties hold:
\begin{itemize}
\item $p_0\in \partial_v L(x,v_0)$,
\item $v_0\in \partial_p  H(x,p_0)$,
\end{itemize}
where we denote by $\partial_v L(x,v_0)$ \big(resp. $\partial_p  H(x,p_0)$\big) the subdifferential (in the sense of convex analysis) of the convex function $v\mapsto L(x,v)$ at $v_0$ (resp. $p\mapsto H(x,p)$ at $p_0$).
\item The Hamiltonian can be recovered from the Lagrangian through the Legendre transform:
$$\forall (x,p)\in T^*M, \quad H(x,p) = \max_{v\in T_xM} p(v)-L(x,v).$$ 
\end{itemize}
\end{prop}
For every $t>0$ we define the action functional $h_t : M\times M\to \R$:
$$ h_t(x,y) = \inf\left\{ \int_{-t}^0 \Big[L\big(\gamma (s),\dot\gamma(s)\big) +c_H\Big]ds,\ \gamma\in AC([-t,0],M) , \gamma(-t)=x, \gamma(0)=y \right\}, $$
where $AC([-t,0],M)$ is the set of absolutely continuous curves from $[-t,0]$ to $M$.

Note that the infimum is actually a minimum that is reached by a Lipschitz curve by the Clarke-Vinter Theorem \cite[Theorem 16.18]{Cl13} or \cite{Cl7}.

We recall for later use a simple property of the action functional that follows from its definition:

\begin{prop}\label{inegtr}
The functions $h_t$ verify the following triangular inequality:
$$\forall (x,y,z,t,t')\in M\times M\times M\times (0,+\infty)\times (0,+\infty), \quad h_{t+t'}(x,z)\leqslant h_t(x,y)+h_t(y,z).$$
\end{prop}

The following characterization hold (see \cite{DavZav, Fa,FS05}):
\begin{prop}\label{sub}
A function $w : M\to \R$ is a critical subsolution of \eqref{HJ-crit} if and only if 
$$\forall (x,y)\in M\times M,\ \forall t>0,\quad w(y) - w(x) \leqslant h_t(x,y).$$
\end{prop}

The Peierls barrier is defined as follows
\begin{df}\label{Peierl}\rm
The Peierls barrier is the function $h : M\times M\to \R$:
$$\forall(x,y)\in M\times M,\quad h(x,y) = \liminf_{t\to+\infty} h_t(x,y).$$
\end{df}

By results of Fathi \cite{Fa1} extended to less regular setting in \cite{DS} the liminf above is actually a limit as soon as $H$ is strictly convex.
Fundamental properties of $h$ are summed up below
\begin{prop}\label{proph}
The Peierls barrier verifies the following properties:
\begin{enumerate}
\item It is finite valued and Lipschitz continuous.
\item If $w:M\to \R$ is a critical subsolution then 
$$\forall (x,y)\in M\times M,\quad w(y) - w(x) \leqslant h(x,y).$$
\item For any $x\in M$, the function $h(x,\cdot)$ is a critical solution.
\item For any $y\in M$, the function $-h(\cdot,y)$ is a critical subsolution.
\end{enumerate}
\end{prop}
We may now give a proper definition of the projected Aubry set:
\begin{df}\label{aubrydef}\rm
The projected Aubry set is the closed and non empty set
$$\mathcal A = \{x\in M,\ \ h(x,x)=0\}.$$
\end{df}
Let us recall the already introduced Theorem \ref{strict}
\begin{Th*}
There exists a function $v : M\to \R$ that is a subsolution of \eqref{HJ-crit}, that is $C^{\infty}$ and strict on $M\setminus \A$ meaning that
$$\forall x\in M\setminus \A , \quad H(x,D_x v) < c_H.$$
\end{Th*}
In particular, it is transparent from the previous Theorem and the definition of the critical value $c_H$ that $\A$ should be non--empty.

As the set of critical solutions is invariant by addition of constant, the critical equation cannot verify a comparison principle. However, the Aubry set allows to have a result in this direction:
\begin{Th}\label{compA}
If $u$ and $v$ are respectively a critical sub and supersolution such that $u\leqslant v$ on $\A$ then $u\leqslant v$ on the whole of $M$.

In particular, $\A$ is a uniqueness set for \eqref{HJ-crit} meaning that if two critical solutions coincide on $\A$ then they are equal.
\end{Th}
We finish this paragraph on further characterizations of critical solutions (see \cite{FS05,DavZav}):
\begin{Th}\label{weakKAM}
The following are equivalent:
\begin{enumerate}
\item The function $u: M\to \R$ is a critical solution.
\item The function $u$ is a critical subsolution and 
$$\forall x\in M,\forall t>0, \exists y\in M,\quad u(x) = u(y) +h_t(y,x).$$
\item\label{calibr} The function $u$ is a critical subsolution and for all $x\in M$, there exists a Lipschitz curve $\gamma_x : (-\infty , 0]\to M$ such that $\gamma(0)=x$ and
$$\forall t>0, \quad u(x) = u\big(\gamma_x(-t)\big) + h_t\big(\gamma_x(-t),x\big).$$
\end{enumerate}
 
\end{Th}

Let us comment on this last point and make precise a Lipschitz constant for $\gamma_x$. Using the triangular inequality \ref{inegtr} and the definition of $h_t$ one infers that 
for $s<t\leqslant 0$,
\begin{equation}\label{calib}
u\big(\gamma_x(t)\big) - u\big(\gamma_x(s)\big) = \int_s^t\Big[L\big(\gamma_x (\sigma),\dot\gamma_x(\sigma)\big) +c_H\Big]d\sigma.
\end{equation}
Recall that $u$ is Lipschitz with constant $\kappa_{c_H}$ given by \eqref{kappa}. Moreover, by superlinearity of $L$, the constant
$$A(\kappa_{c_H}) = \max\left\{ (\kappa_{c_H}+1)\|v\|_x - L(x,v)-c_H , \ \ (x,v)\in TM  \right\}  $$
is such that 
 $$\forall (x,v)\in TM,\quad L(x,v) +c_H\geqslant (\kappa_{c_H}+1)\|v\|_x -A(\kappa_{c_H}).$$
 One then computes using \eqref{calib} that
 \begin{multline*}
 \kappa_{c_H}d\big(\gamma_x(t),\gamma_x(s)\big) 
  \geqslant \int_s^t\left[ (\kappa_{c_H}+1)\|\dot\gamma_x(\sigma)\|_{\gamma_x(\sigma)} -A(\kappa_{c_H})  \right]d \sigma. \\
  \geqslant (\kappa_{c_H}+1)d\big(\gamma_x(t),\gamma_x(s)\big) - (t-s)A(\kappa_{c_H})  .
 \end{multline*}
 Hence we conclude that $d\big(\gamma_x(t),\gamma_x(s)\big) \leqslant (t-s)A(\kappa_{c_H})$ and that $\gamma_x$ is indeed Lipschitz with constant $A(\kappa_{c_H})$.
 
 \subsection{Mather measures}
 We gather here informations already present in \cite{DFIZ} and references therein.
 
 Throughout the rest of this work, we will be dealing with Borel Probability measures on $TM$. We will denote this set by $\PP(TM)$ and. 
 
 We say that a sequence of Borel Probability measures on $TM$, $(\mu_n)_{n\in \mathbb N}$ weakly converges to $\mu\in \PP(TM)$ if for all continuous function $f : TM\to \R$ with compact support,
 $$\int_{TM} f(x,v)d \mu_n(x,v) \to \int_{TM} f(x,v)d \mu(x,v).$$
 If $\mu \in \PP(TM)$, we will denote $\pi^*\mu$ its push forward on $M$ defined by 
 $\pi^*\mu (B) = \mu\big(\pi^{-1}(B)\big)$ for any Borel set $B\subset M$. Recall that $\pi : TM\to M$ is the canonical projection. Note that the following formula hold for all continuous function $f : M\to \R$:
 $$\int_M f(x) d \pi^*\mu(x) = \int_{TM} f\circ \pi(x,v) d \mu(x,v).$$
 
 We then introduce closed measures:
 \begin{df}\label{closed}\rm
 We say that $\mu\in \PP(TM)$ is closed if
 \begin{enumerate}
 \item it has finite first moment: $\int_{TM} \|v\|_x d\mu(x,v)<+\infty$,
 \item for all function $f\in C^1(M,\R)$, $\int_{TM}D_xf (v) d\mu(x,v)=0$.
 \end{enumerate}
 We will denote by $\PP_0(TM)$ the set of closed measures on $TM$.
 \end{df}

Let us  recall the link between closed measures and the critical value $c_H$ (\cite[Theorem A.7.]{DFIZ}). Recall that in this section, $H$ and $L$ are convex and superlinear in the fibers:
\begin{Th}\label{minimizing}
The following holds:
$$\min_{\mu\in \PP_0(TM)} \int_{TM}L(x,v)d\mu(x,v) = -c_H.$$
\end{Th}
Measures realizing the minimum in the previous expression are called minimizing measures, or Mather measures. We will denote by $\PP_{\mathrm{min}}$ the set of Mather measures. Closely linked is the Mather set:
\begin{df}\label{mather}\rm
The Mather set is defined as
$$\widetilde \MM = \bigcup_{\mu \in \PP_{\mathrm{min}}}\mathrm{supp}(\mu)$$
where supp$(\mu)$ denotes the support of $\mu$.

The projected Mather set is $\MM = \pi\big(\widetilde \MM\big)$.
\end{df}
It can be proved that the Mather set is compact and it is not empty as we will see in the last section.

We end by a classical result of weak KAM and Aubry--Mather theories. We provide a proof in this continuous setting for sake of completeness:
\begin{prop}\label{inclusion}
The following inclusion holds: $\MM\subset \A$.
\end{prop}
\begin{proof}
We argue by contradiction and assume that there exists $\mu\in \PP_{\mathrm{min}}$ and $(x_0,v_0)\in \mathrm{supp}(\mu)$ such that $x_0\notin \A$. Let $O$ and $O'$ be open sets of $M$ such that $\overline O \cap \overline{O'}=\varnothing$, $x_0\in O$ and $\A\subset O'$. Let $w : M\to\R$ be the function given by Theorem \ref{strict} and let $\varepsilon>0$ such that 
$H\big(x,D_x w)<c_H-\varepsilon$ for $x\in O$.  Let $\beta : M\to [0,+\infty)$ be a continuous function such that $\beta^{-1}(\{0\})= \A$, $O\subset \beta^{-1}(\{\varepsilon\})$ and 
$$\forall x\in M\setminus \A, \quad H(x,D_x w)<c_H-\beta(x).$$
 By assumption, $\mu\big(\pi^{-1}(O)\big)>0$ so let $\varepsilon'>0$ be such that $2\varepsilon' < \varepsilon \mu\big(\pi^{-1}(O)\big)$. 
 Finally, we apply Proposition \ref{approx} to the function $w$ and Hamiltonian $G(x,p)= H(x,p)-\beta(x)-c_H$ to obtain a $C^1$ function $w_{\varepsilon'}$ such that 
 $$\forall x\in M, \quad H(x,D_x w_{\varepsilon'})<c_H-\beta(x)+\varepsilon'.$$
 
 We then use the Fenchel inequality (Proposition \ref{fenchel}) and the fact that $\mu$ is closed and minimizing to infer
 \begin{align*}
 0=\int_{TM} D_x w_{\varepsilon'} (v) d \mu(x,v) & \leqslant \int_{TM}\left[ H(x,D_x  w_{\varepsilon'})+L(x,v)\right]d \mu(x,v) \\
 &\leqslant  \int_{TM} (c_H-\beta(x)+\varepsilon')d\mu(x,v) -c_H \\
 &\leqslant \int_{\pi^{-1}(O)} -\beta(x) d\mu(x,v) + \int_{TM}(c_H+\varepsilon') d\mu(x,v) - c_H \\
 &= -\varepsilon \mu\big(\pi^{-1}(O)\big) +\varepsilon' <0.
 \end{align*}
This is absurd and proves the proposition.
\end{proof}

\section{The degenerate discounted equation}\label{sec:deg}

We now turn to \eqref{HJ-l} and study its properties.

\subsection{Generalities on the degenerate discounted equation}
 In this section, unless stated otherwise, we fix a $\lambda>0$ and start with a comparison principle that follows from the previous results:

\begin{prop}\label{comp-l}
Equation \eqref{HJ-l} enjoys the two following properties:
\begin{enumerate}
\item \label{comp-1} If $u : M\to \R$ and $v: M\to \R$ are respectively a subsolution and a supersolution of \eqref{HJ-l} then $u\leqslant v$.
\item \label{existence} There exists a unique solution $u_\lambda$ to \eqref{HJ-l}.

\end{enumerate}
\end{prop}

\begin{proof}
\begin{enumerate} \item
Let $w : M\to \R$ be the function given by Theorem \ref{strict}. The function $\tilde w = w-\|w\|_\infty-2\leqslant -2$ is then negative and still  is a subsolution of \eqref{HJ-crit}, that is $C^{\infty}$ and strict on $M\setminus \A$ meaning that
$$\forall x\in M\setminus \A , \quad H(x,D_x \tilde w) < c_H.$$
 Let $\delta = \frac12 \min\{\alpha(x) , \ \ x\in \A\}$. Then $\delta>0$ and $\A \subset \alpha^{-1}\big( (\delta,+\infty)\big)$.
 Let $\eta>0$ such that
  $$\forall x\in \alpha^{-1}[0,\delta] , \quad H(x,D_x \tilde w) < c_H-\eta.$$
  Finally, let $\beta : M\to [0,+\infty)$ be a continuous function such that $\beta^{-1}\{0\} = \A$, $\beta = \eta$ on  $\alpha^{-1}([0,\delta])$ and 
  $$\forall x\in M\setminus \A , \quad H(x,D_x \tilde w) < c_H-\beta(x).$$
  
  If $0<\varepsilon<\min(1,\lambda\delta, \eta)$, by Proposition \ref{approx}, there exists a smooth $w_\varepsilon : M\to \R$ such that $\|\tilde w-w_\varepsilon\|_\infty<\varepsilon$ and 
   $$\forall x\in M , \quad H(x,D_x \tilde w_\varepsilon) \leqslant c_H-\beta(x)+\varepsilon.$$
   It follows that if $x\in \alpha^{-1}([0,\delta])$ then 
   $$\lambda \alpha(x) w_\varepsilon (x) + H(x,D_x \tilde w_\varepsilon)\leqslant  H(x,D_x \tilde w_\varepsilon)\leqslant c_H-\beta(x)+\varepsilon = c_H-\eta+\varepsilon <c_H.$$
   
   On the other hand, if $x\in \alpha^{-1}\big( (\delta , +\infty)\big) $ then 
   $$\lambda \alpha(x) w_\varepsilon (x) + H(x,D_x \tilde w_\varepsilon)\leqslant  -\lambda \delta + H(x,D_x \tilde w_\varepsilon)\leqslant -\lambda\delta +c_H + \varepsilon <c_H.$$
   
   It follows that $w_\varepsilon$ is a smooth strict subsolution to \eqref{HJ-l} and we may apply directly Theorem \ref{compG} that yields the result.
   
   \item This is now a standard result in viscosity solutions theory (see \cite[Théorème 2.12, Théorème 2.14]{barles}). Uniqueness is a direct consequence of the first part. As for existence, it follows from Perron's method. Let $u_0 : M\to \R$ be a critical solution. One verifies that the function $\bar u = u_0+\|u_0\|_\infty \geqslant 0$ is a supersolution to \eqref{HJ-l} and that $\underline u = u_0-\|u_0\|_\infty \leqslant 0$ is a subsolution to \eqref{HJ-l}. By the first point any subsolution $u$ verifies $u\leqslant \bar u$. Moreover, the set $\mathcal S_\lambda$  of subsolutions $u$ to \eqref{HJ-l} such that $\underline u \leqslant u \leqslant \bar u$ is not empty as it contains $\underline u$. The set $\mathcal S_\lambda$ is made of equiLipschitz functions as they verify for almost every $x\in M$,
   $$H(x,D_x u ) \leqslant c_H - \lambda\alpha(x) u(x)\leqslant c_H +\lambda \|\alpha\|_\infty \|\underline u\|_\infty. $$
   A common Lipschitz constant for elements of $\mathcal S_\lambda$ is 
   $$\kappa_{ c_H +\lambda \|\alpha\|_\infty \|\underline u\|_\infty}= \max\big\{ \|p\|_x, \ \ (x,p)\in H^{-1}\big((-\infty, c_H +\lambda \|\alpha\|_\infty \|\underline u\|_\infty ) \big)\big\}.$$
   Hence one verifies that the function $u_\lambda = \max\limits_{u\in \mathcal S_\lambda} u$ is well defined, Lipschitz and a solution to \eqref{HJ-l}.
   
   \end{enumerate}
\end{proof}

Before the next and fundamental subsection let us give a property on the family of solutions:
\begin{lemma}\label{boundedLip}
The family $(u_\lambda)_{\lambda \in (0,1]}$ is equibounded and equiLipschitz.
\end{lemma}
\begin{proof}
We have already seen that $\underline u\leqslant u_\lambda \leqslant \bar u$ for all $\lambda >0$. As far as the equiLipschitz property il concerned, if $0<\lambda \leqslant 1$ a common Lipschitz constant is  
$$\kappa_{ c_H + \|\alpha\|_\infty \|\underline u\|_\infty}= \max\big\{ \|p\|_x, \ \ (x,p)\in H^{-1}\big((-\infty, c_H + \|\alpha\|_\infty \|\underline u\|_\infty ) \big)\big\}.$$
\end{proof}

\subsection{Representation formula for $u_\lambda$}

Thanks to Lemma \ref{boundedLip} we know that for $\lambda<1$ all the functions $u_\lambda$ are equiLipschitz as well as critical solutions. Hence only the knowledge of $H$ on the compact set  
$$\{(x,p)\in T^*M , \quad \|p\|_x\leqslant \kappa_{ c_H + \|\alpha\|_\infty \|\underline u\|_\infty}\}$$
is relevant for our study. Hence up to modifying $H$ outside of this compact set, we assume without loss of generality, until the end of the article,  that $H$ satisfies hypothesis (H2') the superlinearity condition. We will now apply results of section \ref{superlinear}.

We fix once more $\lambda\in(0,1]$ until the end of this section.

\begin{nota}\label{notation}
If $\gamma : I\to M$ is a continuous curve defined on an interval containing $0$ and if $s\in I$ we denote  
$$A_\gamma (s) =  \int_{0}^s \alpha\circ\gamma(\sigma) d\sigma.$$
\end{nota}

We start by a representation formula in the spirit of Theorem \ref{weakKAM}.

\begin{Th}\label{representation}
The function $u_\lambda $ verifies the following properties:
\begin{enumerate}
\item For all $x\in M$ and $t>0$,
\begin{multline}\label{LOl}
u_\lambda(x) = \min_{\substack{\gamma \in AC([-t,0],M   )\\ \gamma(0)=x}} \Bigg \{   \exp\big(\lambda A_\gamma (-t) \big)u_\lambda\big(\gamma(-t)\big) \\
 +\int_{-t}^0 \exp\big(\lambda A_\gamma (s) \big) \left[L\big(\gamma(s),\dot\gamma(s)\big)+c_H\right] ds\Bigg\}.
\end{multline}

\item For all $x\in M$ there exists a Lipschitz curve $\gamma_{\lambda,x} : (-\infty,0] \to M$ with $\gamma_{\lambda,x}(0)=x$ such that for all $t>0$,
\begin{multline}\label{LOl=}
u_\lambda(x) = \exp\left(\lambda A_{\gamma_{\lambda,x}} (-t) \right) u_\lambda\big(\gamma_{\lambda,x}(-t)\big)  \\
+\int_{-t}^0 \exp\left(\lambda A_{\gamma_{\lambda,x}} (s) \right)\left[ L\big(\gamma_{\lambda,x}(s),\dot\gamma_{\lambda,x}(s)\big)+c_H\right] ds.
\end{multline}

\end{enumerate}

\end{Th}

\begin{proof}
We will prove both points simultaneously. The function $u_\lambda $ is a solution to the Hamilton--Jacobi equation 
$$H_\lambda(x,D_x u) =c_H, \quad \quad x \in M,$$
where $H_\lambda (x,p) = \lambda \alpha(x) u_\lambda(x) + H(x,p) $.
The Hamiltonian $H_\lambda$ satisfies hypotheses (H1) and (H2') hence by Proposition \ref{sub} for all $x\in M$, $T>0$ and $\gamma : [-T,0]\to M$ absolutely continuous curve such that $\gamma(0)=x$ we find that for $0\leqslant t\leqslant T$,
$$u_\lambda\big(x)\leqslant u_\lambda\big(\gamma(-t)\big)+ \int_{-t}^0 \Big[L_\lambda\big(\gamma (s),\dot\gamma(s)\big) +c_H\Big]ds,$$
where one computes easily that $L_\lambda$, the Lagrangian associated to $H_\lambda$ is 
$$\forall (y,v)\in TM,\quad L_\lambda(y,v) = L(y,v)-\lambda \alpha(y) u_\lambda(y) .$$
Setting $f(t) =- u_\lambda\big(\gamma(-t)\big)$ one finds by a change of variables $s\mapsto -s$
$$f(t)\leqslant f(0) +\int_0^t  \Big[L\big(\gamma (-s),\dot\gamma(-s)\big) +c_H+ \lambda \alpha\big(\gamma(-s)\big) f(s)\Big]ds. $$
Recalling that $\alpha\geqslant 0$ we may apply Gronwall's inequality (\cite[Lemma 2.1]{mis}), thus obtaining
\begin{multline*}
f(t)\leqslant \exp\left(\lambda \int_{0}^t \alpha\circ\gamma(-\sigma) d\sigma \right)  f(0)\\
 +\int_0^t     \Big[L\big(\gamma (-s),\dot\gamma(-s)\big) +c_H \Big]\exp    \left(\lambda \int_{s}^t \alpha\circ\gamma(-\sigma) d\sigma \right)  ds.   
 \end{multline*}
After dividing by $\exp\left(\lambda \int_{0}^t \alpha\circ\gamma(-\sigma) d\sigma \right) $ and making the changes of variables $\sigma \to -\sigma $ and $s\to -s$ we find that
\begin{equation}\label{ssdeg}
u_\lambda(x) \leqslant \exp\big(\lambda A_\gamma (-t) \big)u_\lambda\big(\gamma(-t)\big) +\int_{-t}^0 \exp\big(\lambda A_\gamma (s) \big) \big[L\big(\gamma(s),\dot\gamma(s)\big)+c_H\big] ds.
\end{equation}

To finish the proof, we will directly establish the second point thus proving that \eqref{LOl} is  valid and that the right hand side is indeed a minimum. We use here item \ref{calib} of Theorem \ref{weakKAM} (with Hamiltonian $H_\lambda$)  to obtain a Lipschitz curve $\gamma_{\lambda,x} : (-\infty , 0] \to M$ such that  $\gamma_{\lambda,x}(0) = x$ and 

\begin{multline*}
u_\lambda(x) - u_\lambda\big(\gamma_{\lambda,x}(-t)\big) = \int_{-t}^0\Big[L_\lambda\big(\gamma_{\lambda,x} (\sigma),\dot\gamma_{\lambda,x}(\sigma)\big) +c_H\Big]d\sigma \\
= \int_{-t}^0\Big[L\big(\gamma_{\lambda,x} (\sigma),\dot\gamma_{\lambda,x}(\sigma)\big) +c_H-\lambda \alpha\big(\gamma_{\lambda,x} (\sigma)\big)    u_\lambda  \big(\gamma_{\lambda,x} (\sigma)\big)  \Big]d\sigma,
\end{multline*}
for all $t>0$.
Hence we find that for $t>0$,
$$g(t) =g(0) + \int_0^t h(s)g(s)ds+\int_0^t \ell(s),$$
where $g(s) =- u_\lambda\big(\gamma_{\lambda,x}(-s)\big) $, $h(s) = \lambda \alpha\big(\gamma_{\lambda,x} (-s)\big)$ and 
$$\ell(s) = L\big(\gamma_{\lambda,x} (-s),\dot\gamma_{\lambda,x}(-s)\big) +c_H.$$
We infer from this,
\footnote{The proof of this last fact follows exactly the classical proof of the Cauchy--Lipschitz Theorem. Fix $T>0$ and let $\mathcal F$ the operator from $\big(C^0([0,T],\R],\|\cdot\|_\infty\big)$ to itself defined  for $f\in C^0([0,T],\R]$ by 
$$\forall t\in [0,T], \quad \mathcal F (f) (t) = g(0) +  \int_0^t h(s)f(s)ds+\int_0^t \ell(s).$$
This functional $\mathcal F$ is $T\|h\|_\infty$-Lipschitz hence for $T$ small enough it is a contraction hence has a unique fixed point. This proves local existence of a solution. As the time of existence is constant, the maximal solution is defined for all $t>0$. Finally one checks by a direct computation that the given function is solution.}

$$\forall t>0, \quad g(t) = g(0)\exp\left(\int_0^t h(s) ds\right) +\int_0^t \ell(s)\exp\left( \int_s^t h(\sigma)d\sigma \right) ds,$$
that is the announced formula.
\end{proof}
\begin{oss}\label{remsg}
Note that, if $0<s<t$ by subtracting \eqref{LOl=} applied to $s$ and $t$ one obtains
\begin{multline}\label{LOlst}
\exp\left(\lambda A_{\gamma_{\lambda,x}} (-s) \right)u_\lambda\big(\gamma_{\lambda,x}(-s)\big) =\exp\left(\lambda A_{\gamma_{\lambda,x}} (-t) \right) u_\lambda\big(\gamma_{\lambda,x}(-t)\big)  \\
+\int_{-t}^{-s} \exp\left(\lambda A_{\gamma_{\lambda,x}} (s) \right)\left[ L\big(\gamma_{\lambda,x}(s),\dot\gamma_{\lambda,x}(s)\big)+c_H\right] ds.
\end{multline}
In particular, the curve $\gamma_{\lambda,x}(-s+\cdot)$ realizes equality in \eqref{LOl=} for the point $\gamma_{\lambda,x}(-s)$.
\end{oss}

We continue by giving  important properties on the  curves $\gamma_{\lambda,x}$ and there asymptotics.

\begin{prop}\label{equiLip}
The curves $\gamma_{\lambda,x}$ with $\lambda\in (0,1]$ and $x\in M$ are equiLipschitz.
\end{prop}
\begin{proof}
Recall that by construction, the curves $\gamma_{\lambda,x}$ verify for $s\leqslant t\leqslant 0$,
\begin{equation}\label{calibl}
u_\lambda\big(\gamma_{\lambda,x}(t)\big) - u_\lambda\big(\gamma_{\lambda,x}(s)\big) = \int_s^t\Big[L_\lambda\big(\gamma_{\lambda,x} (\sigma),\dot\gamma_{\lambda,x}(\sigma)\big) +c_H\Big]d\sigma,
\end{equation}
where $L_\lambda : (x,v)\mapsto L(x,v) - \lambda\alpha(x)u_\lambda(x)$ is the Lagrangian associated to the Hamiltonian $H_\lambda : (x,p)\mapsto h(x,p)+\lambda\alpha(x)u_\lambda(x)$.
Recall that $(u_\lambda)_{\lambda\in (0,1]}$ are equiLipschitz, by Lemma \ref{boundedLip}, with constant that we denote $K>0$. We assume also  by Lemma \ref{boundedLip} that $K$ is big enough such that $\|u_\lambda \|_\infty <K$ for $\lambda \in (0,1]$. Moreover, by superlinearity of $L$, the constant
$$A(K) = \max\left\{ (K+1)\|v\|_x - L(x,v)-c_H , \ \ (x,v)\in TM  \right\}  $$
is such that 
 $$\forall (x,v)\in TM,\quad L(x,v) +c_H\geqslant (K+1)\|v\|_x -A(K).$$
 One then computes using \eqref{calibl} that
 \begin{multline*}
Kd\big(\gamma_{\lambda,x}(t),\gamma_{\lambda,x}(s)\big) 
  \geqslant \int_s^t\left[ (K+1)\|\dot\gamma_{\lambda,x}(\sigma)\|_{\gamma_{\lambda,x}(\sigma)} -A(K) - \lambda \|\alpha\|_\infty \|u_\lambda\|_\infty  \right]d \sigma. \\
  \geqslant (K+1)d\big(\gamma_{\lambda,x}(t),\gamma_{\lambda,x}(s)\big) - (t-s)\big(A(K)+K\|\alpha\|_\infty\big)  .
 \end{multline*}
 Hence we conclude that $d\big(\gamma_{\lambda,x}(t),\gamma_{\lambda,x}(s)\big) \leqslant (t-s)\big(A(K)+K\|\alpha\|_\infty\big)$ and that $\gamma_{\lambda,x}$ is indeed Lipschitz with constant $\big(A(K)+K\|\alpha\|_\infty\big)$ that is independent  of $x\in M$ and $\lambda\in (0,1]$.

\end{proof}

\begin{prop}\label{excursion}
 There exists $T>0$ and $\varepsilon>0$ such that 
for all $x\in M$, $\lambda \in (0,1]$ and $\gamma_{\lambda,x}$ be given by the previous Theorem.
$$\forall t>0, \exists t'\in [t,t+T] , \quad \alpha \circ   \gamma_{\lambda,x}(-t') >\varepsilon.$$

\end{prop}
\begin{proof}
We argue by contradiction and assume that for all $n>0$ there exists $x_n\in M$, $\lambda_n \in (0,1]$ and  $t_n>0$ such that 
$$\forall t\in [t_n,t_n+n] , \quad \alpha \circ   \gamma_{\lambda_n,x_n}(-t) \leqslant \frac1n.$$
We introduce the notations  $\gamma_n =  \gamma_{\lambda_n,x_n}$ and  $A_n = A_{  \gamma_{\lambda_n,x_n}}$ (with respect to Notation \ref{notation}).

Let us define probability  measures $\mu_n$ on $TM$ by
\begin{align*}
\int_{TM} f d\mu_n& =C_n \int_{-(t_n+n)}^{-t_n} \exp\big(\lambda_n A_{n} (s) \big)f\big(\gamma_n(s),\dot\gamma_n(s)\big) ds \\
& =C'_n \int_{-(t_n+n)}^{-t_n} \exp\big(\lambda_n A_n (s) -\lambda_n A_n (-t_n) \big)f\big(\gamma_n(s),\dot\gamma_n(s)\big) ds ,
\end{align*}
 for $f\in C^0(TM,\R)$, where $C_n = \big( \int_{-(t_n+n)}^{-t_n} \exp\big(\lambda_n A_n (s) \big)ds \big)^{-1}$ and 
 \begin{multline*}
 C'_n =  \left( \int_{-(t_n+n)}^{-t_n} \exp\big(\lambda_n A_n (s)-\lambda_n A_n (-t_n) \big)ds \right)^{-1} \\
  =  \left( \int_{-(t_n+n)}^{-t_n} \exp\left(-\lambda_n \int_{s}^{-t_n} \alpha \circ   \gamma_n(\sigma) d\sigma \right)ds \right)^{-1}.
  \end{multline*}
  Note that our hypothesis implies that $C_n' \to 0$. Indeed
$$\forall s\in [-t_n-n,-t_n],\quad  \exp\left(-\lambda_n \int_{s}^{-t_n} \alpha \circ   \gamma_n(\sigma) d\sigma \right)\geqslant \exp\left( \frac{\lambda_n}{n}(t_n+s)\right),
$$
$$(C_n')^{-1} \geqslant \int_{-n}^0 \exp\left( \frac{\lambda_n s}{n}\right)ds = \frac{n}{\lambda_n}\big(1-\exp(-\lambda_n)\big) \geqslant n(1-e^{-1}).
$$

As the $(\gamma_n)_n$ are equiLipschitz (Proposition \ref{equiLip}), the family $(\mu_n)_n$ are supported in a common compact subset of $TM$ and we may extract a weakly convergent subsequence $\mu_{k_n} \rightharpoonup \mu$. We will prove that $\mu $ is a Mather measure.

Let us introduce the notation $B_n(s) = \lambda_n A_n (s)-\lambda_n A_n (-t_n) $.

\textbf{The measure $\mu$ is closed: } let $f : M\to \R$ be a $C^1$ function, one computes (by integration by part):
 
\begin{align*}
 \int_{TM} D_xf(v) d\mu_n(x,v) &=C'_n \int_{-t_n-n}^{-t_n} \exp\big(B_n(s) \big)D_{\gamma_n(s)}f\big(\dot\gamma_n(s)\big) ds  \\
 &= C'_n \Big[ \exp\big(B_n(s)\big)  f\big( \gamma_n(s) \big) \Big]_{-t_n-n}^{-t_n} \\
& \qquad - C'_n  \int_{-t_n-n}^{-t_n} \frac{d}{ds} \Big(\exp\big(B_n(s) \big)  \Big)f\big( \gamma_n(s) \big)ds
 \end{align*}
The first term goes to $0$ as $C_n' \to 0$ and the functions $s\mapsto \exp\big(B_n(s)   \big)f\big( \gamma_n(s) \big)$ are equibounded for $s\leqslant 0$.

To handle the second term we infer
$$ \int_{-t_n-n}^{-t_n} \frac{d}{ds} \Big(\exp\big(B_n (s) \big)  \Big)f\big( \gamma_n(s) \big)ds = 
 \int_{-t_n-n}^{-t_n} \lambda_n \alpha\big(\gamma_n(s)\big) \exp\big(B_n(s) \big) f\big( \gamma_n(s) \big)ds.
$$
Hence, recalling our hypothesis,
$$\left|  \int_{-t_n-n}^{-t_n} \frac{d}{ds} \Big(\exp\big(B_n (s) \big)  \Big)f\big( \gamma_n(s) \big)ds\right| \leqslant  \frac{\lambda_n\|f\|_\infty}{n} \int_{-t_n-n}^{-t_n} \exp\big(B_n (s) \big)  ds= \frac{\lambda_n\|f\|_\infty}{nC_n'},$$
 it follows that
$$\lim_{n\to +\infty} C_n'  \int_{-t_n-n}^{-t_n} \frac{d}{ds} \Big(\exp\big(B_n(s) \big)  \Big)f\big( \gamma_n(s) \big)ds = 0.$$
Finally, taking the limit along the subsequence $k_n$, we conclude that
 $$\int_{TM} D_xf(v) d\mu(x,v)=0,$$
  hence $\mu $ is closed.

\textbf{The measure $\mu$ is minimizing: } indeed let us start from the equalities coming from Theorem \ref{representation} and Remark \ref{remsg}:
\begin{align*}
 \int_{TM} L(x,v) d\mu_n(x,v) & =  C'_n \int_{-t_n-n}^{-t_n} \exp\big(B_n(s) \big)L\big(\gamma_n(s),\dot\gamma_n(s)\big) ds \\
&=C_n'\Big(u_{\lambda_n}\big(\gamma_n(-t_n)\big)- \exp\big(B_n (-t_n-n) \big)   u_{\lambda_n}\big(\gamma_n(-t_n-n)\big)\Big) \\
&\qquad -\int_{TM}c_H d\mu_n.
\end{align*}
Once more, the first term converges to $0$ as $C_n'\to 0$ and the functions $(u_\lambda)_{\lambda\in (0,1]}$ and  $s\mapsto \exp\big(B_n(s)\big), s<0$ are equibounded (Lemma \ref{boundedLip}). Going along the subsequence $k_n$ we conclude that $\int_{TM}L(x,v)d\mu = -c_H$ as announced.

\textbf{The projected measure $\pi^*\mu$ is supported in $\alpha^{-1}(\{0\})$:} indeed, let $\varepsilon>0$ and $\chi_\varepsilon$ the indicatrix function of $\alpha^{-1}\big((\varepsilon,+\infty)\big) $. If $n>0$ is big enough such that  to verify $1/n<\varepsilon$ then clearly 
$$
\int_M \chi_\varepsilon d\pi^*\mu_n =  C'_n \int_{-t_n-n}^{-t_n} \exp\big(B_n(s) \big)\chi_\varepsilon\big(\gamma_n(s)\big) ds=0.
$$

This is a contradiction as the support of $\pi^*\mu$ should be included in the projected Aubry set by Proposition \ref{inclusion}. Hence we have proved the Proposition.
\end{proof}

As a first consequence we deduce

\begin{lemma}\label{courbe-aub}
Let $\lambda\in (0,1]$ and $x\in M$. Then $\int_{-\infty}^0 \alpha\big(\gamma_{\lambda,x} (s)\big)ds = +\infty$.
\end{lemma}

\begin{proof}
Thanks to the previous proposition there exists $\varepsilon>0$ and an increasing sequence $t_n\to +\infty$ such that
$ \alpha\big(\gamma_{\lambda,x}(-t_n)\big) \geqslant \varepsilon$, for all $n\in \mathbb N$. 

 Up to extracting, we may without loss of generality assume that $t_{n+1}-t_n \geqslant 2$ for all $n\in \mathbb N$. As $\gamma_{\lambda,x}$ is Lipschitz, there exists $1>\tau>0$ such that 
$$\forall n\in \mathbb N, \forall h\in (-\tau,\tau),\quad \alpha\big(\gamma_{\lambda,x}(-t_n+h)\big) \geqslant \frac{\varepsilon}{2}.$$
It follows that 
$$\int_{-\infty}^0 \alpha\big(\gamma_{\lambda,x} (s)\big)ds\geqslant \sum_{n=0}^{+\infty} \int_{t_n-\tau}^{t_n+\tau} \alpha\big(\gamma_{\lambda,x} (s)\big)ds \geqslant \tau\varepsilon \sum_{n=0}^{+\infty}1 = +\infty.$$
\end{proof}

As an immediate corollary, we obtain the most important result of this section:

\begin{Th}\label{representation-l}
Let $x\in M$ and $\gamma_{\lambda,x}$ be a curve given by Theorem \ref{representation}, then
\begin{equation}\label{LOl=+}
u_\lambda(x) =\int_{-\infty}^0 \exp\left(\lambda A_{\gamma_{\lambda,x}} (s) \right)\left[ L\big(\gamma_{\lambda,x}(s),\dot\gamma_{\lambda,x}(s)\big)+c_H\right] ds.
\end{equation}
\end{Th}

\section{The convergence result}\label{sec:proof}

We now have all the tools necessary to prove Theorem \ref{main}. The proof follows the original arguments of \cite{DFIZ}. We will start by proving the result in two easy cases before turning to the general case.

\subsection{Warm up: two baby cases}

Though we state it in our setting, the first result holds without the convexity assumption:
\begin{prop}\label{bb1}
Assume that constant functions are critical subsolutions. Then the family $(u_\lambda)_{\lambda>0}$ is nondecreasing as $\lambda \searrow 0$ in particular, it converges. 
\end{prop}

\begin{proof}
As constants are subsolutions it follows that if $\lambda >0$, 
$$\forall x\in M, \quad \lambda \alpha(x)\times 0 +H(x, D_x 0) \leqslant c_H.$$
It follows from the comparison principle (Proposition \ref{comp-l}) that $u_\lambda \geqslant 0$.
If now $0<\lambda< \lambda'$, we have in the viscosity sense
$$ \lambda \alpha(x) u_{\lambda'}(x) +H(x, D_x u_{\lambda'}) \leqslant \lambda' \alpha(x) u_{\lambda'}(x) +H(x, D_x u_{\lambda'})  = c_H.$$
Hence $u_{\lambda'}$ is a subsolution of \eqref{HJ-l} and again by the comparison principle, $u_{\lambda'}\leqslant u_\lambda$. The rest is an immediate consequence of Lemma \ref{boundedLip}.
\end{proof}

The second case reduces to \cite{DFIZ} and can be seen as a motivation for the rest of our work:

\begin{prop}\label{bb2}
Assume that $\alpha(x)>0$ on $M$, then the family $(u_\lambda)_{\lambda>0}$ converges as $\lambda\to 0$. 
\end{prop}
\begin{proof}
If $\lambda>0$ then $u_\lambda$ is also a viscosity solution to 
$$\lambda u_\lambda(x) +H_{\alpha} ( x, D_x u_\lambda) = 0 , \qquad x\in M$$
where $H_\alpha (x,p) = \alpha^{-1}(x) \big(H(x,p) - c_H\big)$ still verifies hypotheses (H1) and (H2). Moreover, any critical solution of $H(x, D_x u)=c_H$ is a solution of $H_\alpha(x,D_x u)=0$ hence the critical constant of $H_\alpha$ is $0$. The result is now a direct application to $H_\alpha$ of \cite[Theorem 1.1]{DFIZ}.
\end{proof}

\begin{oss}\rm
This second case is implied by our hypothesis ($\alpha2$) when the projected Aubry set verifies $\A=M$. This happens for instance in the presence of an exact KAM torus. If $\A=M$ then there is a unique weak KAM solution up to a constant and the selection problem reduces to selecting a particular constant. However, in a forthcoming work (\cite{CFZZ}), we will weaken this hypothesis ($\alpha2)$. For example, in the case of existence of an exact KAM torus where the Hamiltonian flow is conjugated to an irrational rotation, the conclusions of the selection Theorem hold true for any nonnegative, non identically $0$ function $\alpha$. To be more precise, in \cite{CFZZ} we will tackle degenerate discounted problems involving Hamiltonians depending nonlinearly in the value of the function. In contrast with the present work, the proofs are different as they do not rely on explicit representation formulas of the $u_\lambda$.
\end{oss}

\subsection{The general case}

 As the family $(u_\lambda)_{\lambda \in (0,1]}$ is relatively compact, we only have to prove that there is only one accumulation point as $\lambda \to 0$. We start by establishing a constraint on such accumulation points:

\begin{prop}\label{constraints}
Let $(\lambda_n)_n\in \mathbb N$ be a decreasing sequence converging to $0$ such that $(u_{\lambda_n})_n$ uniformly converges to a function $v_0$. Then 
\begin{equation}\label{inegconstraint}
\forall \mu\in \PP_{\mathrm{min}}, \quad \int_{TM} \alpha(x)v_0(x) d\mu(x,v) \leqslant 0.
\end{equation}
\end{prop}
 
 \begin{proof}
Let $\mu$ be a Mather measure. Let $\varepsilon >0$, by applying Proposition \ref{approx} to the Hamiltonian $H(x,p)+\lambda\alpha(x) u_\lambda(x)- c_H$ we find a $C^1$ function $v_\varepsilon : M\to \R$ such that 
$$\forall x\in M,\quad \lambda\alpha(x) u_\lambda(x) + H(x, D_x v_\varepsilon )-c_H \leqslant  \varepsilon.$$
We then integrate this inequality  and apply the Fenchel inequality and the properties of $\mu \in\PP_{\mathrm{min}}$ to infer that 
\begin{align*}
\int_{TM} \lambda\alpha(x) u_\lambda(x)d\mu(x,v) & = \int_{TM}\Big[ \lambda\alpha(x) u_\lambda(x) +D_x v_\varepsilon (v) -L(x,v)-c_H\Big]d\mu(x,v) \\
&\leqslant  \int_{TM}\Big[ \lambda\alpha(x) u_\lambda(x) +H(x, D_x v_\varepsilon)-c_H\Big]d\mu(x,v) \\
&\leqslant  \int_{TM}\varepsilon d\mu(x,v) = \varepsilon.
\end{align*}
As this holds for all $\varepsilon >0$ we conclude, after dividing by $\lambda >0$, that 
$$ \int_{TM} \alpha(x)u_\lambda(x) d\mu(x,v) \leqslant 0.$$
Passing to the limit along the sequence $\lambda_n$, recalling that $u_{\lambda_n} \to u$ uniformly, we conclude that $ \int_{TM} \alpha(x)v_0(x) d\mu(x,v) \leqslant 0$ as was to be shown.

 \end{proof}

\begin{df}\rm\label{limit}
Let us set $\S_0$ be the set of critical subsolutions $u$ to \eqref{HJ-crit} verifying the constraint
\begin{equation}\label{constraint}
\forall \mu\in \PP_{\mathrm{min}}, \quad \int_{TM} \alpha(x)u(x) d\mu(x,v) \leqslant 0.
\end{equation}

We then define $u_0(x) = \sup\limits_{u\in \S_0} u(x)$. 
\end{df}
By stability of the notion of viscosity solutions any accumulation point $v_0$ of $(u_\lambda)_\lambda$ is a critical solution, hence by Proposition \ref{constraints}, $v_0 \in \S_0$. 

The function $u_0$ is well defined as functions in $\S_0$ are equiLipschitz and must take a nonpositive value. By Proposition \ref{viscositygen} $u_0$ is a critical subsolution. Moreover, if $v_0$ is as above, $v_0\leqslant u_0$. To prove the convergence result, we will establish the reverse inequality.

We will need the following
\begin{lemma}\label{chepalle}
Let $\lambda\in (0,1]$ and $x\in M$. Let $\gamma_{\lambda,x} : (-\infty , 0] \to M$ be given by Theorem \ref{representation}, then
$\int_{-\infty}^0 \exp\left(\lambda A_{\gamma_{\lambda,x}} (s) \right)ds <+\infty$.
\end{lemma}

\begin{proof}
We use Proposition \ref{excursion} and consider  $T>0$ and $\varepsilon>0$ such that 
$$\forall t>0, \exists t'\in [t,t+T] , \quad \alpha \circ   \gamma_{\lambda,x}(-t') >\varepsilon.$$
By induction we construct an increasing sequence $(t_n)_n$ such that $t_0\in [1,T+1]$,  $2\leqslant t_{n+1}-t_n \leqslant T+2$ for all $n\in \N$ and $\alpha \circ   \gamma_{\lambda,x}(-t') >\varepsilon$. As in Lemma \ref{courbe-aub}
using that $\gamma_{\lambda,x}$ is Lipschitz, there exists $1>\tau>0$ such that 
$$\forall n\in \mathbb N, \forall h\in (-\tau,\tau),\quad \alpha\big(\gamma_{\lambda,x}(-t_n+h)\big) \geqslant \frac{\varepsilon}{2}.$$
Hence if $n\geqslant 0$,
$$A_{\gamma_{\lambda,x}} (-t_n-\tau) \leqslant -\sum_{k=0}^n \int_{-t_k-\tau}^{-t_k+\tau}\alpha\big(\gamma_{\lambda,x}(h)\big)dh \leqslant -(n+1)\tau \varepsilon,
$$
and using the monotonicity of $A_{\gamma_{\lambda,x}}$:
\begin{align}
\int_{-t_n-\tau}^0 \exp\left(\lambda A_{\gamma_{\lambda,x}} (s) \right)ds& = \int_{-t_0-\tau}^0  \exp\left(\lambda A_{\gamma_{\lambda,x}} (s) \right)ds+ \sum_{k=0}^{n-1} \int_{-t_{k+1}-\tau}^{-t_k-\tau}  \exp\left(\lambda A_{\gamma_{\lambda,x}} (s) \right)ds \nonumber\\
&\leqslant  \int_{-t_0-\tau}^0  \exp\left(\lambda A_{\gamma_{\lambda,x}} (s) \right)ds+ \sum_{k=0}^{n-1}(t_{k+1}-t_k) \exp\left(\lambda A_{\gamma_{\lambda,x}} (-t_k-\tau) \right) \nonumber \\
&\leqslant  \int_{-t_0-\tau}^0  \exp\left(\lambda A_{\gamma_{\lambda,x}} (s) \right)ds+ \sum_{k=0}^{n-1}(T+2) \exp\big(-\lambda (k+1)\tau\varepsilon\big) \nonumber\\
&\leqslant T+2+\frac{(T+2)\exp\big(-\lambda \tau\varepsilon\big)}{1-\exp\big(-\lambda \tau\varepsilon\big)} .\label{eqpalle}
\end{align}
\end{proof}

For $x\in M$, $\lambda\in (0,1]$ we fix $\gamma_{\lambda,x} : (-\infty , 0] \to M$ be given by Theorem \ref{representation} and define the probability measure $\mu_x^\lambda$ on $TM$ defined by
$$\forall f\in C^0(M,\R),\quad \int_{TM} f d\mu_x^\lambda = C_{\lambda,x}\int_{-\infty}^0 \exp\left(\lambda A_{\gamma_{\lambda,x}} (s) \right)f\big(\gamma_{\lambda,x}(s),\dot\gamma_{\lambda,x}(s)\big) ds ,$$
where $(C_{\lambda,x})^{-1} = \int_{-\infty}^0 \exp\left(\lambda A_{\gamma_{\lambda,x}} (s) \right)ds$.

As a corollary of the proof of Lemma \ref{chepalle} we deduce
\begin{lemma}\label{pallebis}
The function $(\lambda,x)\mapsto \lambda (C_{\lambda,x})^{-1}$ is uniformly bounded as $\lambda \to 0$.
\end{lemma}

\begin{proof}
We have seen, \eqref{eqpalle}, that there are constants $T$, $\varepsilon$, and $\tau $, independent  of $x\in M$ and $\lambda\in (0,1]$ such that 
$$ \lambda (C_{\lambda,x})^{-1} \leqslant \lambda \left( T+2+\frac{(T+2)\exp\big(-\lambda \tau\varepsilon\big)}{1-\exp\big(-\lambda \tau\varepsilon\big)}  \right).$$
The right-hand side is bounded as $\lambda \to 0$ and converges to $(T+2)(\tau\varepsilon)^{-1}$.
\end{proof}

\begin{prop}\label{mesure-lim}
The measures of the family $(\mu_x^\lambda)_{\lambda\in (0,1]}$ have support included in a common compact subset of $TM$. Hence it is a relatively compact family in $\PP(TM)$.  Moreover if $\lambda_n \to 0$ is such that $(\mu_x^{\lambda_n})_{n\in \mathbb N}$ converges to $\mu$ then $\mu \in\PP_{\mathrm{min}}$ is a Mather measure.
\end{prop}

\begin{proof}
The first part of the Proposition is a direct consequence of Proposition \ref{equiLip}.

We turn to the second part of the Proposition:

Note that $\lim\limits_{\lambda\to 0}C_{\lambda,x} = 0$. Indeed
 $$(C_{\lambda,x})^{-1} = \int_{-\infty}^0 \exp\left(\lambda A_{\gamma_{\lambda,x}} (s) \right)ds\geqslant \int_{-\infty}^0 \exp\left(\lambda \|\alpha\|_\infty s \right)ds=\frac{1}{\lambda\|\alpha\|_\infty}.$$ 

{\bf The measure $\mu$ is closed:} 
let $f : M\to \R$ be a $C^1$ function, one computes (by integration by part):
 
\begin{align*}
 \int_{TM} D_xf(v) d\mu_{x}^\lambda(x,v) &=C_{\lambda,x} \int_{-\infty}^0 \exp\left(\lambda A_{\gamma_{\lambda,x}} (s) \right)D_{\gamma_{\lambda,x}(s)}f\big(\dot\gamma_{\lambda,x}(s)\big) ds  \\
 &= C_{\lambda,x} \Big[ \exp\left(\lambda A_{\gamma_{\lambda,x}} (s) \right)  f\big( \gamma_{\lambda,x}(s) \big) \Big]_{-\infty}^0 \\
& \qquad - C_{\lambda,x}  \int_{-\infty}^0 \frac{d}{ds} \Big(\exp\left(\lambda A_{\gamma_{\lambda,x}} (s) \right)  \Big)f\big( \gamma_{\lambda,x}(s) \big)ds.
 \end{align*}
 
The first term goes to $0$ as $\lim\limits_{\lambda\to 0}C_{\lambda,x} = 0$ and the function $s\mapsto \exp\left(\lambda A_{\gamma_{\lambda,x}} (s) \right)  f\big( \gamma_{\lambda,x}(s) \big)$ is bounded.

To handle the second term we use that 
$$ \frac{d}{ds} \Big(\exp\left(\lambda A_{\gamma_{\lambda,x}} (s) \right)\Big)=\lambda \alpha\big(\gamma_{\lambda,x}(s)\big) \exp\left(\lambda A_{\gamma_{\lambda,x}} (s) \right)  \geqslant 0.$$ 
Hence
\begin{align*}
\left| \int_{-\infty}^0 \frac{d}{ds} \Big(\exp\left(\lambda A_{\gamma_{\lambda,x}} (s) \right)  \Big)f\big( \gamma_{\lambda,x}(s) \big)ds \right|  & \leqslant  
 \int_{-\infty}^0 \lambda \alpha\big(\gamma_{\lambda,x}(s)\big) \exp\left(\lambda A_{\gamma_{\lambda,x}} (s) \right)  \big|f\big( \gamma_{\lambda,x}(s) \big)\big|ds \\
 &\leqslant \lambda \|\alpha\|_\infty \|f\|_\infty \int_{-\infty}^0 \exp\left(\lambda A_{\gamma_{\lambda,x}} (s) \right) ds\\
 &=\frac{ \lambda \|\alpha\|_\infty \|f\|_\infty}{C_{\lambda,x}} .
\end{align*}
It follows that
$$\lim_{\lambda\to 0} C_{\lambda,x}  \int_{-\infty}^0 \frac{d}{ds} \Big(\exp\left(\lambda A_{\gamma_{\lambda,x}} (s) \right)  \Big)f\big( \gamma_{\lambda,x}(s) \big)ds = 0.$$
Finally, taking the limit along the subsequence $\lambda_n$, we conclude that $\int_{TM} D_xf(v) d\mu(x,v)=0$ hence $\mu $ is closed.

\textbf{The measure $\mu$ is minimizing: } indeed let us start from the equalities coming from Theorem \ref{representation-l}:
\begin{align*}
 \int_{TM} L(x,v) d\mu_x^\lambda(x,v) & =  C_{\lambda,x} \int_{-\infty}^0 \exp\left(\lambda A_{\gamma_{\lambda,x}} (s) \right)L\big(\gamma_{\lambda,x}(s),\dot\gamma_{\lambda,x}(s)\big) ds \\
&=C_{\lambda,x}u_\lambda(x) -\int_{TM}c_H d\mu_x^\lambda.
\end{align*}
Once more, the first term converges to $0$ as $C_{\lambda,x}\to 0$ as $\lambda \to 0$, and the functions $(u_\lambda)_{\lambda} \in (0,1]$   are equibounded. Going along the subsequence $\lambda_n$ we conclude that $\int_{TM}L(x,v)d\mu = -c_H$ as announced.

\end{proof}

\begin{lemma}\label{ineglambda}
Let $w$ be any critical subsolution. For any $\lambda \in (0,1]$ and $x\in M$
$$u_\lambda(x) \geqslant w(x) -\frac{\lambda}{C_{\lambda,x}}\int_{TM} \alpha(y)w(y)d\mu_\lambda^x(y,v).$$
\end{lemma}

\begin{proof}
Let $\varepsilon>0$ and $w_\varepsilon \in C^1(M,\R)$ given by Proposition \ref{approx} such that $\|w-w\|_\varepsilon \leqslant \varepsilon$ and $H(y, D_yw_\varepsilon) \leqslant c_H+\varepsilon$ for all $y\in M$.

Let $t>0$, by Theorem \ref{representation}, and the Fenchel inequality (Proposition \ref{fenchel}),
\begin{align*}
u_\lambda(x) &= \exp\left(\lambda A_{\gamma_{\lambda,x}} (-t) \right) u_\lambda\big(\gamma_{\lambda,x}(-t)\big)+\int_{-t}^0 \exp\left(\lambda A_{\gamma_{\lambda,x}} (s) \right)\left[ L\big(\gamma_{\lambda,x}(s),\dot\gamma_{\lambda,x}(s)\big)+c_H\right] ds  
\\
&\geqslant  \exp\left(\lambda A_{\gamma_{\lambda,x}} (-t) \right) u_\lambda\big(\gamma_{\lambda,x}(-t)\big) \\
&\qquad +\int_{-t}^0  \exp\left(\lambda A_{\gamma_{\lambda,x}} (s) \right)\left[ D_{\gamma_{\lambda,x}(s)}w_\varepsilon \big( \dot\gamma_{\lambda,x}(s)\big)- H\big(\gamma_{\lambda,x}(s),D_{\gamma_{\lambda,x}(s)}w_\varepsilon\big)+c_H\right] ds  
\\
&\geqslant  \exp\left(\lambda A_{\gamma_{\lambda,x}} (-t) \right) u_\lambda\big(\gamma_{\lambda,x}(-t)\big) \\
&\qquad +\int_{-t}^0  \exp\left(\lambda A_{\gamma_{\lambda,x}} (s) \right)\left[ D_{\gamma_{\lambda,x}(s)}w_\varepsilon \big( \dot\gamma_{\lambda,x}(s)\big) \right] ds 
 -\varepsilon \int_{-t}^0  \exp\left(\lambda A_{\gamma_{\lambda,x}} (s) \right)ds
 \\
 &=w_\varepsilon(x) -\int_{-t}^0 \frac{d}{ds} \left[   \exp\left(\lambda A_{\gamma_{\lambda,x}} (s) \right)  \right]w_\varepsilon \big(\gamma_{\lambda,x}(s)\big)ds \\
&\qquad+  \exp\left(\lambda A_{\gamma_{\lambda,x}} (-t) \right) \big[u_\lambda\big(\gamma_{\lambda,x}(-t)\big)-w_\varepsilon\big(\gamma_{\lambda,x}(-t)\big) \big] 
 -\varepsilon \int_{-t}^0  \exp\left(\lambda A_{\gamma_{\lambda,x}} (s) \right)ds \\
  &=w_\varepsilon(x) -\int_{-t}^0   \exp\left(\lambda A_{\gamma_{\lambda,x}} (s) \right)  \lambda \alpha  \big(\gamma_{\lambda,x}(s)\big)w_\varepsilon \big(\gamma_{\lambda,x}(s)\big)ds \\
&\qquad+  \exp\left(\lambda A_{\gamma_{\lambda,x}} (-t) \right) \big[u_\lambda\big(\gamma_{\lambda,x}(-t)\big)-w_\varepsilon\big(\gamma_{\lambda,x}(-t)\big) \big] 
 -\varepsilon \int_{-t}^0  \exp\left(\lambda A_{\gamma_{\lambda,x}} (s) \right)ds .
\end{align*}
Sending $t\to +\infty$ and recalling Lemmas \ref{courbe-aub} and \ref{chepalle} we obtain that
$$u_\lambda(x)\geqslant 
w_\varepsilon(x) -\int_{-\infty}^0   \exp\left(\lambda A_{\gamma_{\lambda,x}} (s) \right)  \lambda \alpha  \big(\gamma_{\lambda,x}(s)\big)w_\varepsilon \big(\gamma_{\lambda,x}(s)\big)ds 
 -\varepsilon \int_{-\infty}^0  \exp\left(\lambda A_{\gamma_{\lambda,x}} (s) \right)ds .
$$

Finally, letting $\varepsilon \to 0$ yields
\begin{multline*}
u_\lambda(x)\geqslant 
w(x) -\int_{-\infty}^0   \exp\left(\lambda A_{\gamma_{\lambda,x}} (s) \right)  \lambda \alpha  \big(\gamma_{\lambda,x}(s)\big)w\big(\gamma_{\lambda,x}(s)\big)ds \\
 =w(x) -\frac{\lambda}{C_{\lambda,x}}\int_{TM} \alpha(y)w(y)d\mu_\lambda^x(y,v).
 \end{multline*}
\end{proof}

We finally finish the proof of Theorem \ref{main}:
\begin{proof}[end of the proof of Theorem \ref{main}]
Recall that we have set $\S_0$  to be the set of critical subsolutions $u$ to \eqref{HJ-crit} verifying the constraint \eqref{constraint}:
$$\forall \mu\in \PP_{\mathrm{min}}, \quad \int_{TM} \alpha(x)u(x) d\mu(x,v) \leqslant 0.$$

We then have defined $u_0(x) = \sup\limits_{u\in \S_0} u(x)$.

Let $(\lambda_n)_{n\in \mathbb N}$ be a decreasing sequence converging to $0$ such that $u_{\lambda_n}$ uniformly converge to a function $v_0$. We have established that $v_0\leqslant u_0$. Let us establish the reverse inequality.

Let $x\in M$. Up to taking a subsequence, we may assume that the family of measures $\mu_{\lambda_n}^x$ converges to a measure $\mu_x$ that is a Mather measure thanks to Proposition \ref{mesure-lim}. Let $w\in \S_0$. We know $\lim\limits_{n\to +\infty } \int_{TM}\alpha(y) w(y)d \mu_{\lambda_n}^x(y,v) =  \int_{TM} \alpha(y)w(y)d \mu \leqslant 0$ and thanks to Lemma \ref{pallebis} we infer that 
$$\limsup_{n\to +\infty}\frac{\lambda_n}{C_{\lambda_n,x}}\int_{TM} \alpha(y)w(y)d\mu_{\lambda_n}^x(y,v) \leqslant 0.$$
Combining with Lemma \ref{ineglambda} we deduce that 
$$v_0(x) \geqslant w(x) - \limsup_{n\to +\infty}\frac{\lambda_n}{C_{\lambda_n,x}}\int_{TM} \alpha(y)w(y)d\mu_{\lambda_n}^x(y,v)\geqslant w(x).$$
As the above inequality holds for all $w\in \S_0$ it follows that $v_0(x) \geqslant u_0(x)$ which concludes the proof as this is true for all $x\in M$.

\end{proof}

\section{An alternate formula for $u_0$}\label{formula}
We want to establish another formula for the limit function $u_0$. Again, this follows closely  section 4 of \cite{DFIZ}.

\begin{df}\rm \label{limit2}
We define the function $\hat u_0 : M\to \R$ by
\begin{equation}\label{limitbis}
\forall x\in M, \quad \hat u_0(x) = \min_{\mu \in \PP_{\mathrm{min}}} \quad \frac{\int_{TM} \alpha(y)h(y,x) d\mu(y,v)}{\int_{TM} \alpha(y) d\mu(y,v)},
\end{equation}
where $h$ is the Peierls barrier (Definition \ref{Peierl}). 
\end{df}
We aim at proving that $u_0 = \hat u_0$.

\begin{lemma}\label{critsub}
The function $\hat u_0$ is a critical subsolution. 
\end{lemma}

\begin{proof}
As $h$ is bounded, $\hat u_0$ is clearly well defined as $\mu\in \PP_{\mathrm{min}}$ is supported on $\alpha^{-1}\big( (0,+\infty)\big)$  (Proposition \ref{inclusion}).

By Proposition \ref{proph} each function $h_y = h(y,\cdot)$ is a critical solution. Hence if $\mu\in  \PP_{\mathrm{min}} $ the function 
$$h_\mu : x\mapsto \frac{\int_{TM} \alpha(y)h(y,x) d\mu(y,v)}{\int_{TM} \alpha(y) d\mu(y,v)},$$
is itself a critical subsolution as a convex combination of such (Proposition \ref{viscositygen}).

Finally, again thanks to Proposition \ref{viscositygen}, $\hat u_0$ is a critical subsolution as a pointwise minimum of critical subsolutions (that are automatically equiLipschitz in this case).
\end{proof}

\begin{lemma}\label{inegu0}
For all $x\in M$, $u_0(x) \leqslant \hat u_0(x)$.
\end{lemma}

\begin{proof}
Let $x\in M$ and $\mu \in  \PP_{\mathrm{min}}$ recall that by Proposition \ref{constraints}, 
$$\int_{TM}\alpha(y) u_0(y) d\mu (y,v)\leqslant 0.$$

Moreover, by  Proposition \ref{proph},
$$\forall y\in M, \quad u_0(x) \leqslant u_0(y) +h(y,x).$$
Multiply the above inequality by $\alpha(y)$ and integrate with respect to $\mu$ to obtain
\begin{align*}
u_0(x)\int_{TM} \alpha(y) d\mu(y,v)& \leqslant \int_{TM}\alpha(y) u_0(y) d\mu(y,v) + \int_{TM} \alpha(y) h(y, x)d\mu(y,v)\\
&\leqslant  \int_{TM} \alpha(y) h(y, x)d\mu(y,v).
\end{align*}
Dividing by $\int_{TM} \alpha(y) d\mu(y,v)$ yields $u_0(x)\leqslant h_\mu (x)$ and taking the minimum over all $\mu \in  \PP_{\mathrm{min}}$ proves the lemma.
\end{proof}

\begin{Th}\label{egaliteu0}
For all $x\in M$, $ \hat u_0(x) = u_0(x)$. 
\end{Th}

\begin{proof}
Let $y\in M$, recall that the function $-h(\cdot , y)$ is a critical subsolution by Proposition \ref{proph}.
Clearly the function $w : M\to \R$ defined by
\begin{align*}
\forall x\in M, \quad w(x) &= -h(x,y) -\max_{\mu \in  \PP_{\mathrm{min}}}\frac{ \int_{TM} -\alpha(z) h(z, y)d\mu(z,v)}{\int_{TM} \alpha(z) d\mu(z,v)   }\\
&= -h(x,y) +\min_{\mu \in  \PP_{\mathrm{min}}}\frac{ \int_{TM} \alpha(z) h(z, y)d\mu(z,v)}{\int_{TM} \alpha(z) d\mu(z,v)   },
\end{align*}
is a critical subsolution verifying the constraint \eqref{constraint}: $w\in \S_0$.
It follows from the definition of $u_0$ that $u_0\geqslant w$.

Evaluating at $y\in \A$ and recalling that $h(y,y)= 0$ by Definition \ref{aubrydef} we obtain that
$$u_0(y) \geqslant \min_{\mu \in  \PP_{\mathrm{min}}}\frac{ \int_{TM} \alpha(z) h(z, y)d\mu(z,v)}{\int_{TM} \alpha(z) d\mu(z,v)   }=\hat u_0(y).$$
As $u_0$ is a critical solution (hence supersolution) and the inequality holds for all $y\in \A$, we conclude from Theorem \ref{compA} that the inequality $u_0\geqslant \hat u_0$ is valid on all $M$ and this with Lemma \ref{inegu0} concludes the proof.

\end{proof}

As a last result we come back to one of the easy cases treated at the beginning of this section:
\begin{prop}
Assume that constant functions are critical subsolutions. Then
$$\forall x\in M,\quad u_0(x)=\min_{y\in A} h(y,x).$$
\end{prop}
\begin{proof}
As constant functions are critical subsolutions, by Proposition \ref{proph}, we get that the Peierls barrier is nonnegative: $h\geqslant 0$.

Let us set $v : x\mapsto \min\limits_{y\in A} h(y,x)$. By Proposition \ref{proph}, $v$ is the minimum of critical solutions of the type $h_y$. It follows from Proposition \ref{viscositygen} that $v$ is itself a critical solution.

As $h\geqslant 0$, each critical solution $h_y$ is also a supersolution of \eqref{HJ-l}. It follows from the comparison principle (Proposition \ref{comp-l}) and from the proof of Proposition \ref{bb1} that $0\leqslant u_\lambda \leqslant h_y$ for all $y\in M$. In particular, $0\leqslant u_\lambda \leqslant v$. Passing to the limit we obtain $0\leqslant u_0\leqslant v$.

To conclude we notice that if $y\in \A$, then $0\leqslant v(y)\leqslant h(y,y)= 0$. Hence both $u_0$ and $v$ are critical solutions that vanish on $\A$. By Theorem \ref{compA}, they are equal.
\end{proof}

As a concluding remark, it is interesting to notice that in this very particular case (when constants are critical subsolutions), the limiting function is actually independent on the function $\alpha$. Of course this is not true in general.

\bibliography{weakly}
\bibliographystyle{siam}

\end{document}